\documentclass[dvipdfmx,11pt]{amsart}
\usepackage{amsmath,amssymb,amsthm}
\usepackage{mathtools,stmaryrd}
\usepackage{color}
\usepackage[utf8]{inputenc}
\usepackage{tikz}
\usetikzlibrary{intersections,calc,arrows}

\theoremstyle{plain}
\newtheorem{thm}{Theorem}[section]
\newtheorem{prop}[thm]{Proposition}
\newtheorem{lem}[thm]{Lemma}
\newtheorem{cor}[thm]{Corollary}

\theoremstyle{definition}
\newtheorem{defn}[thm]{Definition}
\newtheorem{rem}[thm]{Remark}

\numberwithin{equation}{section}

\DeclareFontEncoding{OT2}{}{}
\DeclareFontSubstitution{OT2}{cmr}{m}{n}
\DeclareFontFamily{OT2}{cmr}{\hyphenchar\font45 }
\DeclareFontShape{OT2}{cmr}{m}{n}{%
   <5><6><7><8><9>gen*wncyr%
   <10><10.95><12><14.4><17.28><20.74><24.88>wncyr10}{}
\DeclareFontShape{OT2}{cmr}{b}{n}{%
   <5><6><7><8><9>gen*wncyb%
   <10><10.95><12><14.4><17.28><20.74><24.88>wncyb10}{}

\DeclareMathAlphabet{\mathcyr}{OT2}{cmr}{m}{n}
\DeclareMathAlphabet{\mathcyb}{OT2}{cmr}{b}{n}
\SetMathAlphabet{\mathcyr}{bold}{OT2}{cmr}{b}{n}

\newcommand{\bZ}{\mathbb{Z}}
\newcommand{\bQ}{\mathbb{Q}}
\newcommand{\bR}{\mathbb{R}}
\newcommand{\bC}{\mathbb{C}}
\newcommand{\bF}{\mathbb{F}}

\newcommand{\cA}{\mathcal{A}}

\newcommand{\cF}{\mathcal{F}}
\newcommand{\cI}{\mathcal{I}}
\newcommand{\cRS}{\mathcal{RS}}
\newcommand{\cS}{\mathcal{S}}
\newcommand{\cZ}{\mathcal{Z}}

\newcommand{\frH}{\mathfrak{H}}
\newcommand{\frZ}{\mathfrak{Z}}
\newcommand{\sh}{\mathbin{\mathcyr{sh}}}

\newcommand{\bbra}[1]{\llbracket #1 \rrbracket}

\newcommand{\bk}{{\boldsymbol{k}}}
\newcommand{\bl}{{\boldsymbol{l}}}

\newcommand{\dep}{{\mathrm{dep}}}
\newcommand{\wt}{{\mathrm{wt}}}
\newcommand{\wtd}{{\mathrm{wtd}}}

\title{Weighted sum formula for variants of half multiple zeta values}
\author{Hanamichi Kawamura, Takumi Maesaka, Masataka Ono}

\thanks{The third-named author is supported by Waseda University Grants for Special Research Projects (Grant No. 2022C-300).}

\address[H.~Kawamura]{
Mathematics, Faculty of Science Division I\\
Tokyo University of Science\\ 
1-3 Kagurazaka, Shinjuku-ku, Tokyo, 162-8601\\
Japan}
\email{1121026@ed.tus.ac.jp}

\address[T.~Maesaka]{
Department of Information and Computer Science\\
Kanazawa Institute of Technology\\
7-1, Ohgigaoka, Nonoichi-shi, Ishikawa, 921-8501 \\
Japan}
\email{c1117334@planet.kanazawa-it.ac.jp}

\address[M.~Ono]{
	Global Education Center\\
	Waseda University\\
	1-6-1, Nishi-Waseda, Shinjuku-ku, Tokyo, 169-8050\\
	Japan}
\email{m-ono@aoni.waseda.jp}

\begin{document}

\maketitle

\begin{abstract}
We prove some weighted sum formulas for half multiple zeta values, half finite multiple zeta values, and half symmetric multiple zeta values. The key point of our proof is Dougall's identity for the generalized hypergeometric function ${}_{5}F_{4}$. Similar results for interpolated refined symmetric multiple zeta values and half refined symmetric multiple zeta values are also discussed.
\end{abstract}
%

\section{Introduction}

For positive integers $k_1, \ldots, k_r$ with $k_r \ge2$, multiple zeta value (MZV) $\zeta(k_1, \ldots, k_r)$ is a real number defined by
\begin{equation*}
\zeta(k_1, \ldots, k_r)
\coloneqq
\lim_{N\rightarrow \infty}\zeta^{}_{\le N}(k_1, \ldots, k_r)
\coloneqq
\lim_{N\rightarrow \infty}\sum_{0<n_1<\cdots<n_r\le N}\frac{1}{n^{k_1}_1\cdots n^{k_r}_{r}}.
\end{equation*}
For any function $F$ on the sets of indices and the empty index $\varnothing$, we set $F(\varnothing)$ to be a unit element. A lot of $\bQ$-linear relations among MZVs are known. The weighted sum formula, which is first proved by Ohno and Zudilin \cite{OZ} and there are many variants (for example, \cite{ELO}, \cite{GX} and \cite{Kad}), is one of the most fundamental $\bQ$-relation among MZVs. 

Many variants of multiple zeta values are studied. For an index $\bk=(k_1, \ldots, k_r)$ with $k_r\ge2$ and an indeterminate $t$, Yamamoto \cite{Y} defined the interpolated multiple zeta values ($t$-MZV) $\zeta^{t}(\bk)$ by
\begin{equation*}
\zeta^{t}(\bk)
\coloneqq
\sum_{\bl=(l_1, \ldots, l_s)\preceq \bk}\zeta(\bl)t^{r-s} \in \cZ[t].
\end{equation*}
Here, $\bl \preceq \bk$ means that $\bl$ runs over the all indices of the form $\bl=(k_1 \square \cdots \square k_r)$ in which each $\square$ is filled by the comma , or the plus +, and $\cZ$ is the $\bQ$-algebra generated by all MZVs. Note that $\zeta^{0}(\bk)$ coincides with MZV and $\zeta^{1}(\bk)$ coincides with so called the multiple zeta star value $\zeta^{\star}(\bk)$. Yamamoto proved the sum formula for $t$-MZVs in \cite{Y}, and Li proved a kind of weighted sum formula for $t$-MZVs in \cite{L}. In this paper, we prove a ``$2^{k_r}$-type" weighted sum formula for $1/2$-MZVs, that is, $t$-MZVs with $t=1/2$. For positive integers $i \le r<k$, let $I_i(k,r)$ be the set of indices $\bk=(k_1, \ldots, k_r)$ of $\wt(\bk)=k, \dep(\bk)=r$ and $k_i \ge2$, where $\wt(\bk)$ means the sum of entries and $\dep(\bk)$ is the number of entries.
\begin{thm}\label{thm:wtd SF1}
For positive integers $r<k$, we have
\begin{equation*}
\sum_{\bk \in I_r(k,r)}2^{k_r}\zeta^{\frac{1}{2}}(\bk)
=\frac{1}{2^{r-2}}\biggl\{2^{k-1}-1+\sum_{n=0}^{r-1}\binom{k-1}{n}\biggr\}\zeta(k).
\end{equation*}
\end{thm}

Recently, Kaneko and Zagier defined kinds of variants of MZV, that is, the finite multiple zeta values (FMZVs) and the symmetric multiple zeta values (SMZVs). 
\begin{defn}
We set
\begin{equation*}
\cA \coloneqq 
\left.\biggl(\prod_{\text{$p$ : prime}}\bZ/p\bZ\biggr)
\right/ \biggl(\bigoplus_{\text{$p$ : prime}}\bZ/p\bZ\biggr).
\end{equation*}
Note that since $\bQ$ can be embedded in $\cA$ diagonally, $\cA$ becomes a commutative $\bQ$-algebra.
\end{defn}

\begin{defn}
For an index $\bk=(k_1, \ldots, k_r)$, we define a finite multiple zeta value $\zeta^{}_{\cA}(\bk)$ by
\begin{equation*}
\zeta^{}_{\cA}(\bk) 
\coloneqq (\zeta^{}_{\le p-1}(\bk) \bmod{p})_p \in \cA.
\end{equation*}
\end{defn}

\begin{defn}
For positive integers $k_1, \ldots, k_r$ and $\bullet \in \{\ast, \sh\}$, set
\begin{equation*}
\zeta^{\bullet}_{\cS}(k_1, \ldots, k_r)
\coloneqq
\sum_{i=0}^{r}(-1)^{k_{i+1}+\cdots+k_r}\zeta^{\bullet}(k_1, \ldots, k_i; T)\zeta^{\bullet}(k_r, \ldots, k_{i+1};T).
\end{equation*}
Here, $\zeta^{\bullet}(\bk; T) \in \cZ[T]$ is the $\bullet$-regularized MZVs defined by Ihara, Kaneko and Zagier in \cite{IKZ}. For the detailed definition, see Section 2. It is known that $\zeta^{\bullet}_{\cS}(\bk)$ is independent of the variable $T$ and $\zeta^{\ast}_{\cS}(\bk) \equiv \zeta^{\sh}_{\cS}(\bk) \bmod{\zeta(2)\cZ}$ (for example, see \cite[Proposition 3.2.4]{J} or \cite[Proposition 2.1]{OSY}). We define symmetric multiple zeta value $\zeta^{}_{\cS}(\bk)$ as an element in $\overline{\cZ} \coloneqq \cZ/\zeta(2)\cZ$ by
\begin{equation*}
\zeta^{}_{\cS}(\bk) \coloneqq \zeta^{\bullet}_{\cS}(\bk) \bmod{\zeta(2)\cZ}.
\end{equation*}
\end{defn}

Seki \cite{S}, Murahara and the third-named author \cite{MO} introduced the interpolated FMZVs and SMZVs. For an index $\bk=(k_1, \ldots, k_r)$, an indeterminate $t$ and $\cF \in \{\cA, \cS\}$, we set
\begin{equation*}
\zeta^{t}_{\cF}(\bk)\coloneqq
\sum_{\bl \preceq \bk}\zeta^{}_{\cF}(\bl)t^{r-\dep(\bl)}. 
\end{equation*}
The sum formula for FMZVs and SMZVs was conjectured by Kaneko, and proved by Saito and Wakabayashi for FMZVs \cite[Theorem 1.4]{SW} and by Murahara for SMZVs \cite[Theorem 1.2]{M1}. Moreover, there are several results on weighted sum formula of FMZVs, SMZVs, and their interpolation (for example, \cite{FK}, \cite{HMS2}, \cite{Kam}, \cite{M2}, \cite{MO}). 
In this paper, we also prove a kind of weighted sum formula for $1/2$-FMZVs and $1/2$-SMZVs. Moreover, combining the sum formula for $1/2$-MZV and $1/2$-FMZV, we obtain a ``$2^{k_r}$-type"weighted sum formula of $1/2$-FMZVs and $1/2$-SMZVs.

For a positive integer $k \ge2$, set
\begin{equation}
\frZ_{\cF}(k)
\coloneqq
\begin{cases}
(\frZ_p(k) \bmod{p})_p \coloneqq\biggl(\displaystyle\frac{B_{p-k}}{k} \bmod{p}\biggr)_p \in \cA & \text{if $\cF=\cA$}, \\
\zeta(k) \bmod{\zeta(2)} \in \overline{\cZ} & \text{if $\cF=\cS$}.
\end{cases}
\end{equation}
Here, $B_n \in \bQ$ is the $n$-th Bernoulli number defined by $\displaystyle\frac{te^t}{e^t-1}=\sum_{n=0}^{\infty}\frac{B_n}{n!}t^n$.

\begin{thm}\label{thm:wtd SF2}
For positive integers $1\le i \le r<k$, we have
\begin{equation}\label{eq:wtdSF_F_i}
\sum_{\bk \in I_i(k,r)}(2^{k_i-2}-1)\zeta^{\frac{1}{2}}_{\cF}(\bk)
=\frac{(-1)^i\{(-1)^r-(-1)^k\}}{2^r}\biggl\{2^{k-1}-\sum_{u=0}^{i-1}\binom{k-1}{u}-\sum_{v=0}^{r-i}\binom{k-1}{v}\biggr\}\frZ_{\cF}(k)
\end{equation}
and 
\begin{equation}\label{eq:wtdSF_F_r}
\sum_{\bk \in I_r(k,r)}2^{k_r}\zeta^{\frac{1}{2}}_{\cF}(\bk)
=\frac{1+(-1)^r}{2^{r-2}}\biggl\{2^{k-1}-1+\sum_{u=0}^{r-1}\binom{k-1}{u}\biggr\}\frZ_{\cF}(k)
\end{equation}
\end{thm}
The content of this paper is as follows. In Section 2, we give some generating functions which relate sum formulas for various interpolated MZVs. In Section 3, we evaluate generating functions relating to Theorems \ref{thm:wtd SF1} and Theorem \ref{thm:wtd SF2} for $\cF=\cS$ by using Dougall's identity and then prove Theorem \ref{thm:wtd SF1} and Theorem \ref{thm:wtd SF2} for $\cF=\cS$. The evaluation of the generating function concerning Theorem \ref{thm:wtd SF2} for $\cF=\cA$ is evaluated in Section 4. In the final section, we recall the definition of the refined symmetric multiple zeta values (RSMZVs) introduced by Jarossay and Hirose independently and prove some sum formulas for RSMZV and its variants.

\section{Generating functions of $t$-MZVs, $t$-FMZVs, and $t$-SMZV}

In this section, as a preliminary to prove Theorem \ref{thm:wtd SF1} and Theorem \ref{thm:wtd SF2} for $\cF=\cS$, we calculate generating functions of interpolated MZVs and SMZVs due to the method of Hirose, Murahara, and Saito in \cite{HMS1}.

First, we introduce the notation of the harmonic algebra due to Hoffman \cite{Hof}. Let $\frH$ be the non-commutative polynomial ring over $\bQ$ with two variables $e_0, e_1$, and $\frH^1\coloneqq \bQ+e_1\frH$ be the $\bQ$-subalgebra of $\frH$ consisting of the constant and the words starting in $e_1$, and $\frH^0\coloneqq \bQ+e_1\frH e_0$ be the $\bQ$-subalgebra of $\frH^1$ consisting of constants and the words starting in $e_1$ and ending in $e_0$. Note that $\frH^1$ is generated by $e_{\bk} \coloneqq e_1e^{k_1-1}_{0} \cdots e_1e^{k_r-1}_1$ for all indices $\bk=(k_1, \ldots, k_r)$ and $e_{\varnothing} \coloneqq 1$. Let $Z \colon \frH^1 \to \bR$ be the $\bQ$-linear map sending $e_{\bk}$ to $\zeta(\bk)$. For an index $\bk=(k_1, \ldots, k_r)$, set $S^t(e_{\bk}) \coloneqq \sum_{\bl \preceq \bk}e_{\bl}t^{r-\dep(\bl)} \in \frH^1[t]$, and we extend it to the $\bQ[t]$-linear map $S^{t} \colon \frH^1[t] \to \frH^1[t]$. 

We define the $\bQ$-bilinear map $\sh \colon \frH \times \frH \to \frH$ inductively by the following rule.
\begin{enumerate}
\item $w\sh1=1\sh w=w \ (w \in \frH)$.
\item $(w_1u_1)\sh(w_2u_2)=(w_1\sh w_2u_2)u_1+(w_1u_1\sh w_2)u_2 \ (w_1, w_2 \in \frH, u_1, u_2 \in \{e_0, e_1\})$.
\end{enumerate}
It is known that $\frH_{\sh} \coloneqq (\frH, \sh)$ is a commutative $\bQ$-algebra and $\frH^1_{\sh} \coloneqq (\frH^1, \sh)$ and $\frH^0_{\sh}\coloneqq (\frH^0, \sh)$ are commutative $\bQ$-subalgebras of $\frH_{\sh}$ \cite{Re}.

We define the $\bQ$-bilinear map $\ast \colon \frH^1 \times \frH^1 \to \frH^1$ inductively by the following rule.
\begin{enumerate}
\item $w\ast1=1\ast w=w \ (w \in \frH^1)$.
\item $(w_1e_k \ast w_2e_l)=(w_1\ast w_2e_l)e_k+(w_1e_k\ast w_2)e_l+(w_1\ast w_2)e_{k+l} \ (w_1, w_2 \in \frH^1, k,l \in \bZ_{\ge1})$.
\end{enumerate}
It is known that $\frH^{1}_{\ast} \coloneqq (\frH^1, \ast)$ is a commutative $\bQ$-algebra and $\frH^0_{\ast} \coloneqq (\frH^0, \ast)$ is a commutative $\bQ$-subalgebra of $\frH^{1}_{\ast}$ \cite{Hof}.

For $\bullet \in \{\sh, \ast\}$, it is known that $\frH^1_{\bullet} \cong \frH^0_{\bullet}[e_1]$ as $\bQ$-algebra (see \cite[Theorem 6.1]{Re} and \cite[Theorem 3.1]{Hof}). Then, for any index $\bk$, there exists $w_i \in \frH^0$ satisfying that
\begin{equation*}
e_{\bk}=\sum_{i=0}^{n}\frac{w_i}{i!} \bullet \underbrace{e_1 \bullet \cdots \bullet e_1}_{i}.
\end{equation*}
We set
\begin{equation*}
\zeta^{\bullet}(\bk; T)=Z^{\bullet}_{T}(e_{\bk})
\coloneqq \sum_{i=0}^{n}\frac{Z(w_i)}{i!}T^{i} \in \bR[T].
\end{equation*}
We extend $Z^{\bullet}_{T}$ to the $\bQ$-linear map $Z^{\bullet}_{T} \colon \frH^1 \to \bR[T]$. Finally, we set $\zeta^{t,\bullet}(\bk; T) \coloneqq Z^{\bullet}(S^{t}(e_{\bk})) \in \bR[t, T]$.

Here and in what follows, let $\cI$ denote the set of all indices. For an indeterminate $W$, set
\begin{equation*}
\Gamma_{1, \cI}(W)
\coloneqq
\exp_{\ast}\biggl(\sum_{k=1}^{\infty}\frac{e_k}{k}W^k\biggr) \in \frH^1_{\ast}\bbra{W}, \quad
\Gamma_{1}(W)
\coloneqq
\exp\biggl(\sum_{k=1}^{\infty}\frac{\zeta^{\ast}(k; T)}{k}W^k\biggr) \in \bR[T]\bbra{W},
\end{equation*}
where $\exp_{\ast}$ is the exponential function in $\frH^{1}_{\ast}\bbra{W}$. 
\begin{prop}\label{lem:HMS}
In $\frH^1[A]\bbra{W}$, we have
\begin{equation*}
\sum_{\bk \in \cI}e_{\bk}A^{\dep(\bk)}W^{\wt(\bk)}
=\frac{\Gamma_{1,\cI}(W)}{\Gamma_{1,\cI}((1-A)W)}.
\end{equation*}
\end{prop}

\begin{proof}
See \cite[Proposition 3.1]{HMS1}.
\end{proof}

\begin{prop}\label{prop:St}
In $\frH^1[t, A]\bbra{W}$, we have
\begin{equation*}
\sum_{\bk \in \cI}S^t(e_{\bk})A^{\dep(\bk)}W^{\wt(\bk)}
=\frac{\Gamma_{1,\cI}((1-tA)W)}{\Gamma_{1,\cI}((1-(1-t)A)W)}.
\end{equation*}
\end{prop}

\begin{proof}
We have
\begin{align*}
\sum_{\bk \in \cI}S^t(e_{\bk})A^{\dep(\bk)}W^{\wt(\bk)}
&=\sum_{\bk \in \cI}\sum_{\bl \preceq \bk}
e_{\bl}t^{\dep(\bk)-\dep(\bl)}A^{\dep(\bk)}W^{\wt(\bk)}\\
&=\sum_{\bl \in \cI}e_{\bl}\sum_{\bl \preceq \bk}t^{\dep(\bk)-\dep(\bl)}A^{\dep(\bk)}W^{\wt(\bk)}\\
&=\sum_{\bl \in \cI}e_{\bl}(AW)^{\dep(\bl)}(W+tAW)^{\wt(\bl)-\dep(\bl)}\\
&=\frac{\Gamma_{1, \cI}((1-tA)W)}{\Gamma_{1,\cI}((1-(1-t)A)W)}.
\end{align*}
Here, we use Proposition \ref{lem:HMS} to obtain the final equality.
\end{proof}

\begin{cor}\label{cor:t-ham-zeta}
In $\cZ[t,A, T]\bbra{W}$, we have
\begin{equation*}
\sum_{\bk \in \cI}\zeta^{t, \ast}(\bk; T)A^{\dep(\bk)}W^{\wt(\bk)}
=\frac{\Gamma_{1}((1-tA)W)}{\Gamma_{1}((1-(1-t)A)W)}\exp(ATW).
\end{equation*}
\end{cor}

\begin{proof}
This corollary comes immediately from Proposition \ref{prop:St} and the fact that $Z^{\ast}_{T}$ satisfies the harmonic relation.
\end{proof}

\begin{cor}\label{cor:t-sym}
In $\overline{\cZ}[t,A]\bbra{W}$, we have
\begin{equation*}
\sum_{\bk \in \cI}\zeta^{t}_{\cS}(\bk)A^{\dep(\bk)}W^{\wt(\bk)}
=1.
\end{equation*}
\end{cor}

\begin{proof}
By the definition of $\zeta^{t}_{\cS}(\bk)$ and Corollary \ref{cor:t-ham-zeta}, we have
\begin{align*}
&\sum_{\bk \in \cI}\zeta^{t}_{\cS}(\bk)A^{\dep(\bk)}W^{\wt(\bk)}\\
&=\sum_{\bk, \bl \in \cI}(-1)^{\wt(\bl)}\zeta^{t, \ast}(\bk;T)\zeta^{t, \ast}(\overleftarrow{\bl}; T)A^{\dep(\bk)+\dep(\bl)}W^{\wt(\bk)+\wt(\bl)}\\
&=\frac{\Gamma_{1}((1-tA)W)}{\Gamma_{1}((1-(1-t)A)W)}
\cdot \frac{\Gamma_{1}(-(1-tA)W)}{\Gamma_{1}(-(1-(1-t)A)W)}.
\end{align*}
Here, for an index $\bl=(l_1,\ldots,l_s)$, the symbol $\overleftarrow{\bl}$ stands the index $(l_s,\ldots,l_1)$ (set $\overline{\varnothing} \coloneqq \varnothing$).
Thus, our desired formula follows from the well-known formula in  $\overline{\cZ}\bbra{W}$: 
\begin{equation}\label{eq:rec_gamma}
\Gamma_1(W)\Gamma_1(-W)=\frac{\pi W}{\sin\pi W} \equiv1. \qedhere
\end{equation}
\end{proof}

\begin{defn}
For two indices $\bk, \bl$ and a positive integer $a$, we define an element $K^t\biggl(\begin{matrix} \bk \\ \bl \end{matrix}; a\biggr)$ inductively by
\begin{align*}
&K^t\biggl(\begin{matrix} \bk \\ \varnothing \end{matrix}; a\biggr)
\coloneqq S^{t}(e_{\bk,a}), \qquad 
K^t\biggl(\begin{matrix} \varnothing \\ \bl \end{matrix}; a\biggr)
\coloneqq S^{1-t}(e_{\bl,a}), \\
&K^t\biggl(\begin{matrix} \bk \\ \bl,l \end{matrix}; k\biggr)
+K^t\biggl(\begin{matrix} \bk, k \\ \bl \end{matrix}; l\biggr)
=S^{t}(e_{\bk, k})\ast S^{1-t}(e_{\bl,l}) \ (k, l \in \bZ_{\ge1}).
\end{align*}
Note that $\ast$ is extended $\bQ[t]$-linearly, and if $\bk=(k_1, \ldots, k_r), \bl=(l_1, \ldots, l_s), a\ge2, N \ge1$, we set
\begin{equation*}
\zeta^t_{\le N}\biggl(\begin{matrix} \bk \\ \bl \end{matrix}; a\biggr)
\coloneqq
Z_{\le N}\biggl(K^t\biggl(\begin{matrix} \bk \\ \bl \end{matrix}; a\biggr)\biggr)
=\sum_{\substack{1\le m_1 \le \cdots \le m_r \le q \le N \\ 1\le n_1 \le \cdots \le n_s \le q \le N}}
\frac{t^{r+1-\#\{m_1, \ldots, m_r, q\}}}{m^{k_1}_1\cdots m^{k_r}_r}
\frac{1}{q^a}
\frac{(1-t)^{s+1-\#\{n_1, \ldots, n_s, q\}}}{n^{l_1}_1\cdots n^{l_s}_s}
\end{equation*}
and 
\begin{align*}
\zeta^t\biggl(\begin{matrix} \bk \\ \bl \end{matrix}; a\biggr)
&\coloneqq
\lim_{N \to \infty}Z_{\le N}\biggl(K^t\biggl(\begin{matrix} \bk \\ \bl \end{matrix}; a\biggr)\biggr)
=Z\biggl(K^t\biggl(\begin{matrix} \bk \\ \bl \end{matrix}; a\biggr)\biggr)\\
&=\sum_{\substack{1\le m_1 \le \cdots \le m_r \le q\\ 1\le n_1 \le \cdots \le n_s \le q}}
\frac{t^{r+1-\#\{m_1, \ldots, m_r, q\}}}{m^{k_1}_1\cdots m^{k_r}_r}
\frac{1}{q^a}
\frac{(1-t)^{s+1-\#\{n_1, \ldots, n_s, q\}}}{n^{l_1}_1\cdots n^{l_s}_s}
\end{align*}
which coincides with Schur MZV of anti-hook type as $t \to 0$.
\end{defn}

For a non-empty index $\bl=(l_1, \ldots, l_s)$ and $0 \le j \le s$, $\bl_{[j]} \coloneqq (l_1, \ldots, l_j)$ and $\bl^{[j]} \coloneqq (l_{j+1}, \ldots, l_s)$ and $\varnothing_{[0]}=\varnothing^{[0]} \coloneqq \varnothing$. The same method in \cite{HMS1} proves the next lemmas, so we omit the proof. 

\begin{lem}
For indices $\bk, \bl$ and a positive integer $a$, we have
\begin{equation}
K^t\biggl(\begin{matrix} \bk \\ \overleftarrow{\bl} \end{matrix}; a\biggr)
=\sum_{j=0}^{\dep(\bl)}(-1)^{j}S^t(e_{\bk, a, \bl_{[j]}})\ast S^{1-t}(e_{\overleftarrow{\bl^{[j]}}}).
\end{equation}
\end{lem}

\begin{lem}
For indices $\bk, \bl$ and a positive integer $a$, we have
\begin{equation*}
S^t(e_{\bk, a, \bl})
=\sum_{j=0}^{\dep(\bl)}(-1)^{j}K^t\biggl(\begin{matrix} \bk \\ \overleftarrow{\bl_{[j]}} \end{matrix}; a\biggr)
\ast S^{t}(e_{\bl^{[j]}}).
\end{equation*}
\end{lem}

For an index $\bk=(k_1, \ldots, k_r)$ and a positive integer $N$, we define elements $\zeta^{t}_{\le N}(\bk), \zeta^{t}_{=N}(\bk)$ and $\zeta^{t}_{<N}(\bk)$ of $\bQ[t]$ by
\begin{align*}
\zeta^{t}_{\le N}(\bk)
&\coloneqq 
\sum_{1\le n_1 \le \cdots \le n_r \le N}\frac{t^{r-\#\{n_1, \ldots, n_r\}}}{n^{k_1}_1\cdots n^{k_r}_r}, \\
\zeta^{t}_{< N}(\bk)
&\coloneqq 
\sum_{1\le n_1 \le \cdots \le n_r < N}\frac{t^{r-\#\{n_1, \ldots, n_r\}}}{n^{k_1}_1\cdots n^{k_r}_r}=\zeta^{t}_{\le N-1}(\bk),\\
\zeta^{t}_{=N}(\bk)
&\coloneqq \zeta^{t}_{\le N}(\bk)-\zeta^{t}_{<N}(\bk).
\end{align*}
Note that the usual multiple harmonic sum $\zeta_{\le N}(\bk)=\sum_{1\le n_1<\cdots<n_r\le N}\frac{1}{n^{k_1}_1\cdots n^{k_r}_r}$ coincides with $\zeta^{0}_{\le N}(\bk)$.

For a non-negative integer $N$ and a symbol $X$, set
\begin{equation*}
(X)_N
\coloneqq
\begin{cases}
X(X+1)\cdots(X+N-1) & \text{if $N\ge1$}, \\
1 & \text{if $N=0$}.
\end{cases}
\end{equation*}
\begin{lem}\label{lem:gen_func_fmh}
For positive integer $N$, we have the following identity in $\bQ[t, A]\bbra{W}$:
\begin{equation*}
\sum_{\bk \in \cI}\zeta^{t}_{\le N}(\bk)A^{\dep(\bk)}W^{\wt(\bk)}
=\frac{(1-(1-(1-t)A)W)_N}{(1-(1+tA)W)_N}.
\end{equation*}
\end{lem}

\begin{proof}
Since $\zeta^{}_{\le N}(\bk)$ satisfies the harmonic relation, we obtained the desired formula by Proposition \ref{prop:St} and the equality $\exp(\sum_{k\ge1}\zeta^{}_{\le N}(k)W^k/k)=N!/(1-W)_N$ in $\bQ\bbra{t}$ for any positive integer $N$.
\end{proof}

\section{Proof of Theorem \ref{thm:wtd SF1} and Theorem \ref{thm:wtd SF2} for $\cF=\cS$}

In this section, we prove Theorem \ref{thm:wtd SF1} and Theorem \ref{thm:wtd SF2} for $\cF=\cS$. For indeterminates $A, B, W$, set
\begin{equation*}
G^{\frac{1}{2}}_{\wtd}(A, B, W)
\coloneqq
\sum_{\substack{\bk, \bl \in \cI \\ a \ge2}}(2^{a-2}-1)\zeta^{\frac{1}{2}}
\biggl(
\begin{matrix}
\bk \\
\bl
\end{matrix}
;a
\biggr)
A^{\dep(\bk)}B^{\dep(\bl)}W^{\wt(\bk)+a+\wt(\bl)}
\end{equation*}
and 
\begin{equation*}
F^{\frac{1}{2}}_{\wtd, \cS}(A, B, W)
\coloneqq
\sum_{\substack{\bk, \bl \in \cI \\ a \ge2}}(2^{a-2}-1)\zeta^{\frac{1}{2}}_{\cS}(\bk, a, \bl)
A^{\dep(\bk)}B^{\dep(\bl)}W^{\wt(\bk)+a+\wt(\bl)}.
\end{equation*}
Then, these two generating functions satisfy the following relation.
\begin{lem}\label{lem:decomp}
In $\overline{\cZ}[A, B]\bbra{W}$, we have
\begin{equation*}
F^{\frac{1}{2}}_{\wtd, \cS}(A, B, W)
=G^{\frac{1}{2}}_{\wtd}(A, -B, W)+G^{\frac{1}{2}}_{\wtd}(B, -A, -W).
\end{equation*}
\end{lem}

\begin{proof}
By definition of $F^{\frac{1}{2}}_{\wtd, \cS}(A, B, W)$, we have
\begin{align*}
&F^{\frac{1}{2}}_{\wtd, \cS}(A,B,W)\\
&=\sum_{\substack{\bk, \bl \in \cI \\ a\ge2}}\sum_{j=0}^{\dep(\bl)}
(-1)^{\wt(\bl^{[j]})}(2^{a-2}-1)\zeta^{\frac{1}{2}, \ast}(\bk, a, \bl_{[j]})\zeta^{\frac{1}{2}, \ast}(\overleftarrow{\bl^{[j]}})
A^{\dep(\bk)}B^{\dep(\bl)}W^{\wt(\bk)+a+\wt(\bl)}\\
&\quad+\sum_{\substack{\bk, \bl \in \cI \\ a\ge2}}\sum_{i=0}^{\dep(\bk)}
(-1)^{\wt(\bk^{[i]}, \bl)+a}(2^{a-2}-1)\zeta^{\frac{1}{2}, \ast}(\bk_{[i]})\zeta^{\frac{1}{2}, \ast}(\overleftarrow{\bk^{[i]}, a, \bl})
A^{\dep(\bk)}B^{\dep(\bl)}W^{\wt(\bk)+a+\wt(\bl)}\\
&=\sum_{\substack{\bk, \bl_1, \bl_2 \in \cI \\ a\ge2}}(-1)^{\wt(\bl_2)}
(2^{a-2}-1)\zeta^{\frac{1}{2}, \ast}(\bk, a, \bl_1)\zeta^{\frac{1}{2}, \ast}(\overleftarrow{\bl_2})
A^{\dep(\bk)}B^{\dep(\bl_1)+\dep(\bl_2)}W^{\wt(\bk)+\wt(\bl_1)+\wt(\bl_2)+a}\\
&\quad+
\sum_{\substack{\bk_1, \bk_2, \bl \in \cI \\ a\ge2}}(-1)^{\wt(\bk_2)+\wt(\bl)+a}
(2^{a-2}-1)\zeta^{\frac{1}{2}, \ast}(\bk_1)\zeta^{\frac{1}{2}, \ast}(\overleftarrow{\bk_2, a, \bl})
A^{\dep(\bk_1)+\dep(\bk_2)}B^{\dep(\bl)}W^{\wt(\bk_1)+\wt(\bk_2)+\wt(\bl)+a}.
\end{align*}
Let $X_1$ (resp. $X_2$) denote the first (resp. second) sum above. From Corollary \ref{cor:t-ham-zeta} and \eqref{eq:rec_gamma}, we have
\begin{equation}\label{eq:X_1}
\begin{split}
X_1
&=\sum_{\substack{\bk, \bl \in \cI \\ a\ge2}}
(2^{a-2}-1)\zeta^{\frac{1}{2}, \ast}(\bk, a, \bl)A^{\dep(\bk)}B^{\dep(\bl)}W^{\wt(\bk)+\wt(\bl)}
\frac{\Gamma_1(-(1+\frac{B}{2})W)}{\Gamma_1(-(1-\frac{B}{2})W)}\\
&=\sum_{\substack{\bk, \bl \in \cI \\ a\ge2}}\sum_{j=0}^{\dep(\bl)}(-1)^{j}(2^{a-2}-1)
\zeta^{\frac{1}{2}}
\Bigl(
\begin{matrix}
\bk \\ \overleftarrow{\bl_{[j]}}
\end{matrix}
;a\Bigr)
\zeta^{\frac{1}{2}}(\bl^{[j]})A^{\dep(\bk)}B^{\dep(\bl)}W^{\wt(\bk)+\wt(\bl)+a}
\frac{\Gamma_1(-(1+\frac{B}{2})W)}{\Gamma_1(-(1-\frac{B}{2})W)}\\
&=G^{\frac{1}{2}}_{\wtd}(A, -B, W)\frac{\Gamma_1((1+\frac{B}{2})W)\Gamma_1(-(1+\frac{B}{2})W)}{\Gamma_1((1-\frac{B}{2})W)\Gamma_1(-(1-\frac{B}{2})W)}
\equiv G^{\frac{1}{2}}_{\wtd}(A, -B, W) \pmod{\zeta(2)\cZ}.
\end{split}
\end{equation}
By a similar calculation, we have
\begin{equation}\label{eq:X_2}
X_2 \equiv G^{\frac{1}{2}}_{\wtd}(B, -A, -W) \pmod{\zeta(2)\cZ}.
\end{equation}
From the equations \eqref{eq:X_1} and \eqref{eq:X_2}, we obtain the desired formula.
\end{proof}

\begin{lem}
In $\bR[A, B]\bbra{W}$, we have
\begin{equation*}
G^{\frac{1}{2}}_{\wtd}(A,B,W)
=\sum_{n=1}^{\infty}\frac{W^3(n-W)}{n(n-2W)}
\frac{(1-(1-\frac{A}{2})W)_{n-1}(1-(1-\frac{B}{2})W)_{n-1}}{(1-(1+\frac{A}{2})W)_n(1-(1+\frac{B}{2})W)_n}.
\end{equation*}
\end{lem}

\begin{proof}
By the definition of $F^{\frac{1}{2}}_{\wtd}(A, B, W)$, we have
\begin{align*}
&G^{\frac{1}{2}}_{\wtd}(A, B, W)\\
&=\sum_{\substack{\bk, \bl \in \cI \\ a\ge2}}(2^{a-2}-1)
\zeta^{\frac{1}{2}}
\Bigl(
\begin{matrix}
\bk \\ \bl
\end{matrix}
;a\Bigr)
A^{\dep(\bk)}B^{\dep(\bl)}W^{\wt(\bk+a+\wt)(\bl)}\\
&=\sum_{\substack{\bk, \bl \in \cI \\ a\ge2}}(2^{a-2}-1)\sum_{n=1}^{\infty}
\frac{(\zeta^{\frac{1}{2}}_{<n}(\bk)+\frac{1}{2}\zeta^{\frac{1}{2}}_{=n}(\bk))
(\zeta^{\frac{1}{2}}_{<n}(\bl)+\frac{1}{2}\zeta^{\frac{1}{2}}_{=n}(\bl))}{n^a}
A^{\dep(\bk)}B^{\dep(\bl)}W^{\wt(\bk)+a+\wt(\bl)}\\
&=\sum_{\substack{\bk, \bl \in \cI \\ a\ge2}}\frac{2^{a-2}-1}{4}\sum_{n=1}^{\infty}
\frac{(\zeta^{\frac{1}{2}}_{<n}(\bk)+\zeta^{\frac{1}{2}}_{\le n}(\bk))
(\zeta^{\frac{1}{2}}_{<n}(\bl)+\zeta^{\frac{1}{2}}_{\le n}(\bl))}{n^a}
A^{\dep(\bk)}B^{\dep(\bl)}W^{\wt(\bk)+a+\wt(\bl)}.
\end{align*}
From Lemma \ref{lem:gen_func_fmh}, we have
\begin{align*}
\sum_{\bk \in \cI}(\zeta^{\frac{1}{2}}_{<n}(\bk)+\zeta^{\frac{1}{2}}_{\le n}(\bk))A^{\dep(\bk)}W^{\wt(\bk)}
&=\frac{(1-(1-\frac{A}{2})W)_{n-1}}{(1-(1+\frac{A}{2})W)_{n-1}}+\frac{(1-(1-\frac{A}{2})W)_n}{(1-(1+\frac{A}{2})W)_n}\\
&=\frac{2(n-W)(1-(1-\frac{A}{2})W)_{n-1}}{(1-(1+\frac{A}{2})W)_n}.
\end{align*}
Since
\begin{equation*}
\sum_{a=2}^{\infty}(2^{a-2}-1)\biggl(\frac{W}{n}\biggr)^a
=\frac{W^3}{n(n-W)(n-2W)},
\end{equation*}
we obtain the desired formula.
\end{proof}

To prove Theorems \ref{thm:wtd SF1} and \ref{thm:wtd SF2} for $\cF=\cS$, we need to relate $G^{\frac{1}{2}}_{\wtd}(A, B, W)$ to $\psi_1(W)$ which is a generating function of Riemann zeta values $\zeta(k)$ defined by
\begin{equation}
\psi_1(W)
\coloneqq \sum_{k \ge2}\zeta(k)W^{k-2} \in \cZ\bbra{W}.
\end{equation}
Note that let 
\begin{equation*}
\psi(z) \coloneqq \frac{d}{dz}\log\Gamma(z) 
\end{equation*}
denote the digamma function, then $\psi_1(W)=-\psi(1-W)-\gamma$ ($\gamma$ is the Euler--Mascheroni constant).

\begin{prop}\label{prop:F^1/2_Dougall}
In $\bR\bbra{A,B,W}$, we have
\begin{multline*}
G^{\frac{1}{2}}_{\wtd}(A, B, W)\\
=\frac{1}{2(1-\frac{A}{2})(1-\frac{B}{2})}
\biggl\{
\psi_1(2W)+\psi_1\biggl(\frac{A+B}{2}W\biggr)-\psi_1\biggl(\biggl(1+\frac{A}{2}\biggr)W\biggr)-\psi_1\biggl(\biggl(1+\frac{B}{2}\biggr)W\biggr)
\biggr\}.
\end{multline*}
\end{prop}
This proposition is an immediate consequence of Dougall's identity which relates the special value of the generalized hypergeometric function $ _{5}F_{4}$ and the Gamma function $\Gamma(z)$.

\begin{defn}
For positive integer $r$, $a_1, \ldots, a_{r+1} \in \bC$, and $b_1, \ldots, b_r \in \bC \setminus\bZ_{\le 0}$, the generalized hypergeometric function is defined by
\begin{equation*}
{}_{r+1} F_{r}
\biggl(
\begin{matrix}
a_1 & \cdots & a_{r+1} \\
b_1 & \cdots & b_r
\end{matrix}
; x
\biggr)
\coloneqq
\sum_{n=0}^{\infty}\frac{(a_1)_n \cdots (a_{r+1})_n}{(b_1)_n \cdots (b_r)_n}\frac{x^n}{n!}.
\end{equation*}
This infinite series converges when $|x|<1$. Moreover, this converges at $|x|=1$ if $\Re(\sum_{i=1}^{r}b_i-\sum_{j=1}^{r+1}a_j)>0$.
\end{defn}

\begin{prop}[{Dougall's identity, \cite{D}}]
For $a,b,c,d \in \bC$ with $\Re(1+a-b-c-d)>0$, we have
\begin{align}\label{eq:Dougall}
\begin{split}
&{}_{5} F_{4}
\biggl(
\begin{matrix}
a & 1+\frac{a}{2} & b & c & d \\
\frac{a}{2} & 1+a-b & 1+a-c & 1+a-d & 
\end{matrix}
; 1
\biggr)\\
&\qquad\qquad\qquad\qquad =\frac{\Gamma(1+a-b)\Gamma(1+a-c)\Gamma(1+a-d)\Gamma(1+a-b-c-d)}
{\Gamma(1+a)\Gamma(1+a-b-c)\Gamma(1+a-b-d)\Gamma(1+a-c-d)}
\end{split}
\end{align}
\end{prop}

\begin{proof}[{Proof of Proposition \ref{prop:F^1/2_Dougall}}]
By differentiating the both-hand sides of \eqref{eq:Dougall} with $d$ and taking the limit $d \to 0$, we have and equality in $\bR\bbra{a,b,c}$ : 
\begin{equation}\label{eq:after_diff}
\begin{split}
&\sum_{n=1}^{\infty}\frac{2n+a}{n(n+a)}\frac{(b)_n(c)_n}{(1+a-b)_n(1+a-c)_n}\\
&=-\psi_1(b-a)-\psi_1(c-a)+\psi_1(-a)+\psi_1(b+c-a).
\end{split}
\end{equation}
By setting $a=-2W, b=-(1-\frac{A}{2})W, c=-(1-\frac{B}{2})W$ in \eqref{eq:after_diff}, we obtain the desired formula.
\end{proof}

\begin{proof}[{Proof of Theorem \ref{thm:wtd SF1}}]
Note that
\begin{equation*}
G^{\frac{1}{2}}_{\wtd}(A, 0, W)
=\sum_{1\le r<k}\Biggl\{\sum_{\bk \in I_r(k,r)}(2^{k_r-2}-1)\zeta^{\frac{1}{2}}(\bk)\Biggr\}A^{r-1}W^k.
\end{equation*}
From Proposition \ref{prop:F^1/2_Dougall}, for $k \in \bZ_{\ge2}$, the coefficient of $G^{\frac{1}{2}}_{\wtd}(A, 0, W)$ at $W^k$ is
\begin{equation*}
\frac{1}{2-A}\biggl\{2^{k-1}+\Bigl(\frac{A}{2}\Bigr)^{k-1}-\Bigl(1+\frac{A}{2}\Bigr)^{k-1}-1\biggr\}\zeta(k).
\end{equation*}
Thus, for a positive integer $r<k$, the coefficient of $G^{\frac{1}{2}}_{\wtd}(A, 0, W)$ at $A^{r-1}$ is
\begin{equation}\label{eq:coeff_Ar-1}
\frac{1}{2^{r}}\biggl\{2^{k-1}-1-\sum_{n=0}^{r-1}\binom{k-1}{n}\biggr\}\zeta(k).
\end{equation}
Therefore, Theorem \ref{thm:wtd SF1} is obtained by the equation \eqref{eq:coeff_Ar-1} and the sum formula for $1/2$-MZVs \cite[Theorem 1.1]{Y}.
\end{proof}

\begin{proof}[{Proof of Theorem \ref{thm:wtd SF2} for $\cF=\cS$}]
It is also easy to see that
\begin{equation*}
F^{\frac{1}{2}}_{\wtd, \cS}(A, B, W)
=\sum_{1\le i \le r<k}\Biggl\{\sum_{\bk \in I_i(k,r)}(2^{k_i-2}-1)\zeta^{\frac{1}{2}}_{\cS}(\bk)\Biggr\}A^{i-1}B^{r-i}W^k,
\end{equation*}
From this equation and Theorem \ref{thm:wtd SF1}, the left-hand side of \eqref{eq:wtdSF_F_i} coincides with
\begin{align*}
&\Bigl(\text{the coefficient of $G^{\frac{1}{2}}_{\wtd}(A, -B, W)+G^{\frac{1}{2}}_{\wtd}(B, -A, -W)$ at $A^{i-1}B^{r-i}W^k$}\Bigr)\\
&=(-1)^{r-i}\Bigl(\text{the coefficient of $G^{\frac{1}{2}}_{\wtd}(A, B, W)$ at $A^{i-1}B^{r-i}W^k$}\Bigr)\\
&\quad+(-1)^{i-1}\Bigl(\text{the coefficient of $G^{\frac{1}{2}}_{\wtd}(A, B, W)$ at $A^{r-i}B^{i-1}W^k$}\Bigr)\\
&=\frac{(-1)^i((-1)^r-(-1)^k)}{2^r}\Biggl\{2^{k-1}-\sum_{u=0}^{i-1}\binom{k-1}{u}-\sum_{v=0}^{r-i}\binom{k-1}{v}\Biggr\}\zeta(k),
\end{align*}
which completes the proof of \eqref{eq:wtdSF_F_i}. In the case $i=r$, \eqref{eq:wtdSF_F_r} is proved immediately from \eqref{eq:wtdSF_F_i} and the sum formula for $1/2$-SMZV \cite[Proposition 5.20]{S}.
\end{proof}

\begin{rem}
We define $F^{t}(A, B, W) \in \cZ[t, A, B]\bbra{W}$ and $F^{t}_{\cS}(A,B,W) \in \overline{\cZ}[t, A, B]\bbra{W}$ by
\begin{equation*}
G^t(A, B, W)
\coloneqq \sum_{\substack{\bk, \bl \in \cI \\ a \ge2}}
\zeta^{t}\biggl(\begin{matrix} \bk \\ \overleftarrow{\bl} \end{matrix}; a\biggr)
A^{\dep(\bk)}B^{\dep(\bl)}W^{\wt(\bk)+a+\wt(\bl)}
\end{equation*}
and 
\begin{equation*}
F^{t}_{\cS}(A, B, W)
\coloneqq \sum_{\substack{\bk, \bl \in \cI \\ a \ge2}}
\zeta^{t}_{\cS}(\bk, a, \bl)
A^{\dep(\bk)}B^{\dep(\bl)}W^{\wt(\bk)+a+\wt(\bl)}
\end{equation*}
Then, from a similar argument to Lemma \ref{lem:decomp}, we have a congruence 
\begin{equation}\label{eq:decomp2}
F^{t}_{\cS}(A,B,W)
\equiv G^{t}(A, -B, W)+G^{t}(B, -A, -W)
\end{equation}
in $\overline{\cZ}[t, A, B]\bbra{W}$, and an equality
\begin{equation}\label{eq:gen_G}
G^{t}(A, B, W)
=\sum_{n=1}^{\infty}
\frac{W^2(n-W)}{n}
\frac{(1-(1-(1-t)A)W)_{n-1}}{(1-(1-(1+t)B)W)_n}
\frac{(1-(1-tB)W)_{n-1}}{(1-(1+tA)W)_n}
\end{equation}
in $\bR[t, A, B]\bbra{W}$. Moreover, by differentiating with respect to $b$ and taking the limit $b \to 0$ in Gauss' hypergeometric theorem
\begin{equation*}
_{2}F_1
\biggl(
\begin{matrix}
a & b \\
c & 
\end{matrix}
; 1
\biggr)=\frac{\Gamma(c)\Gamma(c-a-b)}{\Gamma(c-a)\Gamma(c-b)},
\end{equation*}
we have
\begin{equation}\label{eq:Gauss_to_psi}
\sum_{n=1}^{\infty}\frac{1}{n(n+a)}\frac{(1+b)_n}{(c)_n}
=\frac{1}{a}\bigl\{\psi_1(c-a)-\psi_1(c)\bigr\}.
\end{equation}
Therefore, from \eqref{eq:decomp2}, \eqref{eq:gen_G} and \eqref{eq:Gauss_to_psi}, we can prove the sum formulas for the $t$-MZVs first proved by Yamamoto \cite[Theorem 1.1]{Y}, and the $i$-admissible sum formula (the formula expressing the sum over $I_{i}(k,r)$) for the SMZVs proved by Murahara \cite[Theorem 1.2]{M1}, and $r$-admissible sum formula for the $t$-SMZVs proved by Seki \cite[Proposition 5.20]{S}. Moreover, we can give another proof of the sum formula for the anti-hook type Schur multiple zeta values proved first by Bachmann--Kadota--Suzuki--Yamamoto--Yamasaki in \cite[Theorem 3.8]{BKSYY}.
\end{rem}

\section{Proof of Theorem \ref{thm:wtd SF2} for $\cF=\cA$}

In this section, we prove Theorem \ref{thm:wtd SF2} for $\cF=\cA$. First, similar to the case $\cF=\cS$, we consider the generating function. For a positive integer $N$ and indeterminates $A, B, W$, set
\begin{equation*}
G^{\frac{1}{2}}_{\wtd, \le N}(A, B, W)
\coloneqq
\sum_{\substack{\bk, \bl \in \cI \\ a \ge2}}(2^{a-2}-1)\zeta^{\frac{1}{2}}_{\le N}
\biggl(
\begin{matrix}
\bk \\
\bl
\end{matrix}
;a
\biggr)
A^{\dep(\bk)}B^{\dep(\bl)}W^{\wt(\bk)+a+\wt(\bl)}
\end{equation*}
and 
\begin{equation*}
F^{\frac{1}{2}}_{\wtd, \le N}(A, B, W)
\coloneqq
\sum_{\substack{\bk, \bl \in \cI \\ a \ge2}}(2^{a-2}-1)\zeta^{\frac{1}{2}}_{\le N}(\bk, a, \bl)
A^{\dep(\bk)}B^{\dep(\bl)}W^{\wt(\bk)+a+\wt(\bl)}.
\end{equation*}
\begin{lem}\label{lem:GF_fin}
In $\bQ[A, B]\bbra{W}$, we have
\begin{equation*}
F^{\frac{1}{2}}_{\wtd, \le N}(A, B, W)
=G^{\frac{1}{2}}_{\wtd, \le N}(A, -B, W)\frac{(1-(1-\frac{B}{2})W)_N}{(1-(1+\frac{B}{2})W)_N}.
\end{equation*}
\end{lem}

\begin{proof}
This lemma is proved by the same method for the proof of Lemma \ref{lem:decomp}.
\end{proof}

For a positive integer $n$, recall $\frZ_p(n) = \frac{B_{p-n}}{n}$ and $\frZ_{\cA}(n)=(\frZ_p(n) \bmod{p})_{p}$. 

\begin{lem}\label{lem:ZC}
For an indeterminate $a$ and a sufficiently large prime $p$, in $(\bZ/p^2\bZ)\bbra{a}$, we have
\begin{align}
\frac{(1+a)_{p-1}}{(p-1)!}
&=1-p\sum_{n=1}^{\infty}\frZ_p(n+1)a^n, \label{eq:ZC1} \\
\frac{(p-1)!}{(1+a)_{p-1}}
&=1+p\sum_{n=1}^{\infty}\frZ_p(n+1)a^n. \label{eq:ZC2}
\end{align}
In particular, we have
\begin{equation*}
\frac{(1+a)_{p-1}}{(p-1)!} \equiv \frac{(p-1)!}{(1+a)_{p-1}} \equiv 1 \pmod{p}.
\end{equation*}
\end{lem}

\begin{proof}
For a symbol $X$ and a non-negative integer $n$, set $\{X\}^{n} \coloneqq \underbrace{X, \ldots, X}_{n}$. We have
\begin{equation*}
\frac{(1+a)_{p-1}}{(p-1)!}
=\prod_{n=1}^{p-1}\Bigl(1+\frac{a}{m}\Bigr)
=1+\sum_{n=1}^{p-1}\zeta^{}_{\le p-1}(\{1\}^n)a^n.
\end{equation*}
From the result of Zhou--Cai \cite[p.1332]{ZC} and $\frZ_p(n)=0$ for even $n$, we have
\begin{equation}\label{eq:1n}
\zeta^{}_{\le p-1}(\{1\}^n)\equiv p(-1)^{n+1}\frZ_p(n+1)= -p\frZ_p(n+1) \pmod{p^2},
\end{equation}
which gives \eqref{eq:ZC1}. Moreover, we have
\begin{align*}
\frac{(p-1)!}{(1+a)_{p-1}}
=\prod_{n=1}^{p-1}\Bigl(1+\frac{a}{n}\Bigr)^{-1}
=\prod_{n=1}^{p-1}\Bigl(1-\frac{a}{n}+\frac{a^2}{n^2}-\cdots\Bigr)
=1+\sum_{n=1}^{p-1}(-1)^{n}\zeta^{\star}_{\le p-1}(\{1\}^n)a^n,
\end{align*}
where
\begin{equation*}
\zeta^{\star}_{\le p-1}(k_1, \ldots, k_r)
\coloneqq
\sum_{1\le m_1\le \cdots \le m_n\le p-1}\frac{1}{m^{k_1}_1\cdots m^{k_r}_r}
\end{equation*}
for positive integers $k_1, \ldots, k_r$. From \eqref{eq:1n} and and the antipode-like relation for $\zeta^{}_{\le p-1}(\bk)$ and $\zeta^{\star}_{\le p-1}(\bk)$ \cite[Proposition 6]{IKOO}, we have
\begin{equation}\label{eq:1nstar}
\zeta^{\star}_{\le p-1}(\{1\}^n)
=(-1)^{n}\zeta^{}_{\le p-1}(\{1\}^n)
\equiv p(-1)^{n+1}\frZ_p(n+1)= -p\frZ_p(n+1) \pmod{p^2}.
\end{equation}
This completes the proof of \eqref{eq:ZC2}.
\end{proof}

From Lemmas \ref{lem:GF_fin} and \ref{lem:ZC}, for any large prime, we have
\begin{equation}\label{eq:F_to_G}
F^{\frac{1}{2}}_{\wtd, \le p-1}(A, B, W)
\equiv G^{\frac{1}{2}}_{\wtd, \le p-1}(A, -B, W) \pmod{p}.
\end{equation}
Next, we define an element $\psi_{\cA}(x) \in \cA\bbra{x}$ by
\begin{equation*}
\psi_{\cA}(x) \coloneqq \sum_{n=2}^{\infty}\frZ_{\cA}(n)x^n.
\end{equation*}
This $\psi_{\cA}(x)$ is regarded as an $\cA$-analogue of $\psi_1(x)=\sum_{k=2}^{\infty}\zeta(k)x^{k-2}$. We evaluate $G^{\frac{1}{2}}_{\wtd, \le p-1}(A,B,W)$ with respect to $\psi_{\cA}(x)$.

\begin{lem}\label{lem:Dougall_A}
Let $a,b,c$ be indeterminates and $p$ a prime. In $\cA\bbra{a,b,c}$, we have
\begin{align}\label{eq:Dougall_A}
\begin{split}
&\Biggl(\sum_{n=1}^{p-1}\frac{2n+a}{n(n+a)}\frac{(b)_n(c)_n}{(1+a-b)_n(1+a-c)_n}\bmod{p}\Biggr)_p\\
&=\frac{1}{(a-b)(a-c)}\bigl\{(a-b-c)\psi_{\cA}(a)+a\psi_{\cA}(a-b-c)+b\psi_{\cA}(c)+c\psi_{\cA}(b)\\
&\qquad\qquad\qquad\qquad\qquad\qquad\qquad\qquad-(c-a)\psi_{\cA}(b-a)-(b-a)\psi_{\cA}(c-a)\bigr\}.
\end{split}
\end{align}
\end{lem}
\begin{proof}
Let $p$ be a sufficiently large prime. In \eqref{eq:Dougall_A}, set $d=-p$. Since $(-p)_n=0$ for $n\ge p+1$, we have an equality in $\bZ_{(p)}\bbra{a,b,c}$.
\begin{equation}\label{eq:key_A}
\begin{split}
&p\sum_{n=1}^{p-1}\frac{2n+a}{a+n+p}\frac{(1+a)_{n-1}(b)_n(c)_n}{n(1+a-b)_n(1+a-c)_n(1+a+p)_n}\frac{(1-p)_{n-1}}{(n-1)!}\\
&=1-\frac{(1+a)_p(1+a-b-c)_p}{(1+a-b)_p(1+a-c)_p}
-\frac{2p+a}{a}\frac{(a)_p(b)_p(c)_p}{(1+a-b)_p(1+a-c)_p(1+a+p)_p}.
\end{split}
\end{equation}
We consider the equality above by taking $\bmod{p^2}$. In this case, since the right-hand side of \eqref{eq:key_A} is divisible by $p$, each term on the right-hand side of \eqref{eq:key_A} is considered by taking $\bmod{p}$. Thus, in $(\bZ/p^2\bZ)\bbra{a,b,c}$, the equation \eqref{eq:key_A} is equivalent to
\begin{align*}
p\sum_{n=1}^{p-1}\frac{2n+a}{n(n+a)}\frac{(b)_n(c)_n}{(1+a-b)_n(1+a-c)_n}
&=1-\frac{a(a-b-c)}{(a-b)(a-c)}\frac{(1+a)_{p-1}(1+a-b-c)_{p-1}}{(1+a-b)_{p-1}(1+a-c)_{p-1}}\\
&\quad-\frac{bc}{(a-b)(a-c)}\frac{(1+b)_{p-1}(1+c)_{p-1}}{(1+a-b)_{p-1}(1+a-c)_{p-1}}\\
&=-\frac{a(a-b-c)}{(a-b)(a-c)}\biggl\{\frac{(1+a)_{p-1}(1+a-b-c)_{p-1}}{(1+a-b)_{p-1}(1+a-c)_{p-1}}-1\bigg\}\\
&\quad -\frac{bc}{(a-b)(a-c)}\biggl\{\frac{(1+b)_{p-1}(1+c)_{p-1}}{(1+a-b)_{p-1}(1+a-c)_{p-1}}-1\biggr\}.
\end{align*}
By applying Lemma \ref{lem:Dougall_A} for the right-hand side of \eqref{eq:key_A}, we have 
\begin{align*}
&p\sum_{n=1}^{p-1}\frac{2n+a}{n(n+a)}\frac{(b)_n(c)_n}{(1+a-b)_n(1+a-c)_n}\\
&=\frac{p}{(a-b)(a-c)}\biggl\{(a-b-c)\sum_{n=2}^{\infty}\frZ_p(n)a^n+a\sum_{n=2}^{\infty}\frZ_p(n)(a-b-c)^n\\
&\quad+b\sum_{n=2}^{\infty}\frZ_p(n)c^n+c\sum_{n=2}^{\infty}\frZ_p(n)b^n
-(c-a)\sum_{n=2}^{\infty}\frZ_p(n)(b-a)^n-(b-a)\sum_{n=2}^{\infty}\frZ_p(n)(c-a)^n\biggr\}.
\end{align*}
in $(\bZ/p^2\bZ)\bbra{a,b,c}$. Since the equality in $\bF_p\bbra{a,b,c}$ obtained after dividing by $p$ holds for any sufficiently large prime $p$, we obtained the desired formula in $\cA\bbra{a,b,c}$.
\end{proof}

\begin{lem}\label{lem:G_to_psi}
In $\cA\bbra{A,B,W}$, we have
\begin{align*}
&\Bigl(G^{\frac{1}{2}}_{\wtd, \le p-1}(A, B, W)\bmod{p}\Bigr)_{p}\\
&=\frac{W^2}{2(1+\frac{A}{2})(1+\frac{B}{2})}\biggl\{\frac{A+B}{2}\psi_{\cA}(2W)+2\psi_{\cA}\Bigl(\frac{A+B}{2}W\Bigr)+\Bigl(1-\frac{A}{2}\Bigr)\psi_{\cA}\Bigl(\Bigl(1-\frac{B}{2}\Bigr)W\Bigr)\\
&\qquad+\Bigl(1-\frac{B}{2}\Bigr)\psi_{\cA}\Bigl(\Bigl(1-\frac{A}{2}\Bigr)W\Bigr)
+\Bigl(1+\frac{A}{2}\Bigr)\psi_{\cA}\Bigl(\Bigl(1+\frac{B}{2}\Bigr)W\Bigr)
+\Bigl(1+\frac{B}{2}\Bigr)\psi_{\cA}\Bigl(\Bigl(1+\frac{A}{2}\Bigr)W\Bigr)\biggr\}.
\end{align*}
\end{lem}

\begin{proof}
By a similar calculation to that of $G^{\frac{1}{2}}_{\wtd}(A,B,W)$ in the proof of Proposition \ref{prop:F^1/2_Dougall}, we have
\begin{equation*}
G^{\frac{1}{2}}_{\wtd, \le N}(A, B, W)
=\sum^{N}_{n=1}\frac{W^3(n-W)}{n(n-2W)}
\frac{(1-(1-\frac{A}{2})W)_{n-1}}{(1-(1+\frac{A}{2})W)_{n}}
\frac{(1-(1-\frac{B}{2})W)_{n-1}}{(1-(1+\frac{B}{2})W)_n}.
\end{equation*}
Then, by substituting $N=p-1, a=-2W, b=(1+\frac{A}{2})W, c=(1+\frac{B}{2})W$ in \eqref{eq:Dougall_A} and  using $\psi_{\cA}(-X)=-\psi_{\cA}(X)$, we obtained the desired formula.
\end{proof}

\begin{proof}[{Proof of Theorem \ref{thm:wtd SF2} for $\cF=\cA$}]
From \eqref{eq:F_to_G} and Lemma \ref{lem:G_to_psi}, by using the definition of $\psi_{\cA}(X)$ and $\frZ_{\cA}(n)=0$ for any even positive integer $n$, we have
\begin{align}\label{eq:exp_of_F}
\begin{split}
&\Bigl(F^{\frac{1}{2}}_{\wtd, \le p-1}(A, -B, W)\bmod{p}\Bigr)_p\\
&=
\sum_{n=1}^{\infty}\frZ_{\cA}(2n+1)\sum_{l=0}^{n}\sum_{a=0}^{l}\sum_{b=0}^{n-l-1}\binom{2n+1}{2l+1}
\biggl\{\Bigl(\frac{B}{2}\Bigr)^{2a+1}\Bigl(\frac{A}{2}\Bigr)^{2b}+\Bigl(\frac{A}{2}\Bigr)^{2a+1}\Bigl(\frac{B}{2}\Bigr)^{2b}\biggr\}W^{2n+1}.
\end{split}
\end{align}
Let $1\le i \le r<k$. We compare the coefficients of the both-hand side in $\eqref{eq:exp_of_F}$ at $A^{i-1}B^{r-i}W^k$. Since only terms with odd exponent of $W$ in the right-hand side of $\eqref{eq:exp_of_F}$ appear, Theorem \ref{thm:wtd SF2} for $\cF=\cA$ is proved if $k$ is even.

Let consider the case $k$ is odd and set $k=2n+1$. Since the parities of exponent of $A$ and $B$ in the left-hand side of \eqref{eq:conclusion} are different, $r$ must be even. Thus we have
\begin{equation}\label{eq:conclusion}
\sum_{\bk \in I_i(k,r)}(2^{k_i-2}-1)\zeta^{\frac{1}{2}}_{\cA}(\bk)
=
\begin{cases}
\dfrac{1}{2^{r-1}}\frZ_{\cA}(k)\displaystyle\sum_{l=i/2}^{(k-r+i-3)/2}\binom{k}{2l+1} & \text{if $i$ is even}, \\
-\dfrac{1}{2^{r-1}}\frZ_{\cA}(k)\displaystyle\sum_{l=(r-i+1)/2}^{(k-i-2)/2}\binom{k}{2l+1} & \text{if $i$ is odd}.
\end{cases}
\end{equation}
It is easy to see that the rational multiple in the right-hand side of \eqref{eq:conclusion} coincides with
\begin{equation*}
\frac{(-1)^{i}\{(-1)^r-(-1)^k\}}{2^{r}}\biggl\{2^{k-1}-\sum_{u=0}^{i-1}\binom{k-1}{u}-\sum_{v=0}^{r-i}\binom{k-1}{v}\biggr\},
\end{equation*}
we obtain the desired formula.
\end{proof}

\begin{rem}
Set
\begin{align}\label{eq:gen_N}
\begin{split}
F^{t}_{\le N}(A, B, W)
&\coloneqq \sum_{\substack{\bk, \bl \in \cI \\ a \ge2}}\zeta^{t}(\bk, a, \bl)A^{\dep(\bk)}B^{\dep(\bl)}W^{\wt(\bk)+\wt(\bl)+a}\\
&=\sum_{1\le i\le r<k}\biggl\{\sum_{\bk \in I_{i}(k,r)}\zeta^{t}_{\le p-1}(\bk)\biggr\}A^{i-1}B^{r-i}W^{k}.
\end{split}
\end{align}
By a similar calculation, we have
\begin{equation}\label{eq:gen_p-1}
F^{t}_{\le p-1}(A, B, W)
\equiv \sum_{n=1}^{p-1}\frac{W^2(n-W)}{n}
\frac{(1-(1-(1-t)A)W)_{n-1}}{(1-(1-(1-t)B)W)_n}
\frac{(1-(1+tB)W)_{n-1}}{(1-(1+tA)W)_n}\pmod{p}.
\end{equation}
On the other hand, from Gauss' hypergeometric theorem, for a positive integer $N$, we have
\begin{equation*}
\sum_{n=1}^{N}\frac{(a)_n}{n(1-b)_n}\frac{(-N)_n}{n!}
=\frac{(1-a-b)_N}{(1-a)_N},
\end{equation*}
which is equivalent to the Chu--Vandermonde identity. Therefore, by a similar argument, we have
\begin{equation}\label{eq:gen_A}
\Biggl(\sum_{n=1}^{p-1}\frac{(a)_n}{n(1-b)_n}\bmod{p}\Biggr)_p
=\frac{1}{b}\bigl\{\psi_{\cA}(a+b)-\psi_{\cA}(a)-\psi_{\cA}(b)\bigr\}.
\end{equation}
Thus, if we set $a=-W(1-A), b=-W(1-B)$ in \eqref{eq:gen_A}, and $t=0$ in \eqref{eq:gen_N} and \eqref{eq:gen_p-1}, we obtain the $i$-admissible sum formula for FMZV which was first proved by Saito and Wakabayashi \cite{SW}. Moreover, if we set $a=-W+(1-t)AW, b=W+tAW$ in \eqref{eq:gen_A}, we obtain the $r$-admissible sum formula for $t$-FMZV which was first proved by Seki \cite{S}.
\end{rem}

\section{Some sum formulas for refined symmetric multiple zeta values}

In this section, we prove some sum formulas of the interpolated refined symmetric multiple zeta values ($t$-RSMZVs). The refined symmetric multiple zeta value (RSMZV) was introduced first by Jarossay. In \cite{J}, he was introduced an object called the $\Lambda$-adjoint multiple zeta values by using the theory of associators, and the refined symmetric multiple zeta values as the coefficients of the $\Lambda$-adjoint multiple zeta values. Independently, Hirose \cite{Hi} also defined RSMZV using the notion of iterated integrals. 

\begin{defn}[{Interpolated refined SMZV}]
For an index $\bk$, set
\begin{equation*}
\zeta^{t, \ast}(\bk; T)
\coloneqq
\sum_{\bl \preceq \bk}\zeta^{\ast}(\bl; T)t^{\dep(\bk)-\dep(\bl)}
\in \cZ[t, T].
\end{equation*}
We define the interpolated refined symmetric multiple zeta value ($t$-RSMZV) $\zeta^{t}_{\cRS}(\bk) \in \cZ[2\pi\sqrt{-1}][t]$ by
\begin{equation}\label{eq:def_of_tRSMZV}
\zeta^{t}_{\cRS}(\bk)
\coloneqq
\sum_{i=0}^{\dep(\bk)}(-1)^{\wt(\bk^{[i]})}
\zeta^{t, \ast}\biggl(\bk_{[i]}; -\frac{\pi\sqrt{-1}}{2}\biggr)
\zeta^{t, \ast}\biggl(\overleftarrow{\bk^{[i]}}; \frac{\pi\sqrt{-1}}{2}\biggr).
\end{equation}
\end{defn}

\begin{rem}\label{rem:xi_RSMZV}
We define $t$-RSMZV by mimicking the definition of the $\xi$-values defined by Bachmann--Takeyama--Tasaka in \cite{BTT1}. Hirose mentioned in \cite[Remark 13]{Hi} that the refined symmetric multiple zeta value coincides with their $\xi$-value. We can give an equivalent definition of $t$-RSMZV by using the iterated integral expression. 
\end{rem}

\begin{rem}
It is easy to see that 
\begin{equation*}
\zeta^{t}_{\cRS}(\bk)
\equiv
\zeta^{t}_{\cS}(\bk) \bmod{2\pi\sqrt{-1}\cZ[2\pi\sqrt{-1}][t]}
\end{equation*}
in $\cZ[2\pi\sqrt{-1}][t]/2\pi\sqrt{-1}\cZ[2\pi\sqrt{-1}][t] \cong(\cZ[2\pi\sqrt{-1}]/2\pi\sqrt{-1}\cZ[2\pi\sqrt{-1}])[t]\cong\overline{\cZ}[t]$.
\end{rem}

\begin{prop}\label{prop:SFof1/2RS}
In $\cZ[2\pi\sqrt{-1}][A,B]\bbra{W}$, we have
\begin{align*}
&\sum_{\substack{\bk,\bl \in \cI \\ a\ge2}}(2^{a-2}-1)\zeta^{\frac{1}{2}}_{\cRS}(\bk,a,\bl)A^{\dep(\bk)}B^{\dep(\bl)}W^{\wt(\bk)+a+\wt(\bl)}\\
&=F^{\frac{1}{2}}_{\wtd}(A,-B,W)\frac{(2+B)\sin(\pi(1-\frac{B}{2})W)}{(2-B)\sin(\pi(1+\frac{B}{2})W)}\exp(-\pi\sqrt{-1}BW)\\
&\quad+F^{\frac{1}{2}}_{\wtd}(B,-A,-W)\frac{(2+A)\sin(\pi(1-\frac{A}{2})W)}{(2-A)\sin(\pi(1+\frac{A}{2})W)}\exp(\pi\sqrt{-1}AW).
\end{align*}
\end{prop}

\begin{proof}
This proposition is proved in the same way as Lemma \ref{lem:decomp} by using Corollary \ref{cor:t-ham-zeta}.
\end{proof}

For non-negative integers $h \le l$, set
\begin{multline*}
a_{h,l}
\coloneqq
\sum_{0 \le i \le j \le \lfloor l/2 \rfloor}\sum_{u \in \bZ}
(-1)^{l-j+u}\frac{(2-2^{2i})2^{-2j-h}B_{2i}}{(2i)!(2j-2i+1)!(l-2j)!}\\
\times\binom{2j-2i}{u}\binom{2i}{2j-(l-h)-u}(2\pi\sqrt{-1})^{l},
\end{multline*}
where we understand that the binomial coefficient $\binom{a}{b}$ is zero if $a<b$ or $b<0$. Therefore the inner sum is finite and $a_{h,l}$ is well-defined.
\begin{thm}\label{thm:wt SF3}
For positive integers $i \le r<k$, we have
\begin{align}\label{eq:SF3}
\begin{split}
&\sum_{\bk=(k_1, \ldots, k_r) \in I_i(k,r)}(2^{k_i-2}-1)\zeta^{\frac{1}{2}}_{\cRS}(\bk)\\
&=\sum_{0\le h \le l \le k-2}\biggl\{(-1)^{r-i-h}\biggl(2^{k-l-1}-\sum_{u=0}^{i-1}\binom{k-l-1}{u}-\sum_{v=0}^{r-i-h}\binom{k-l-1}{v}\biggl)\\
&\quad+(-1)^{k+i-1-h}\biggl(2^{k-l-1}-\sum_{u=0}^{r-i}\binom{k-l-1}{u}-\sum_{v=0}^{i-1-h}\binom{k-l-1}{v}\biggr)\biggr\}
a_{h,l}\frac{\zeta(k-l)}{2^{r-h}}.
\end{split}
\end{align}
\end{thm}

\begin{proof}
Set
\begin{equation*}
Z_{a,b}(k)
\coloneqq \Biggl\{2^{k-1}-\sum_{u=0}^{a}\binom{k-1}{u}-\sum_{v=0}^{b}\binom{k-1}{v}\Biggr\}\frac{\zeta(k)}{2^{a+b+1}}.
\end{equation*}
Then, by the Taylor expansion of $\sin X, \csc X=1/\sin X$ and $\exp X$, we have
\begin{align*}
\frac{(2+B)\sin(\pi(1-\frac{B}{2})W)}{(2-B)\sin(\pi(1+\frac{B}{2})W)}\exp(-\pi\sqrt{-1}BW)
=\sum_{0 \le h \le k}a_{h,k}B^hW^k.
\end{align*}
Therefore, the coefficient of 
\begin{equation*}
F^{\frac{1}{2}}_{\wtd}(A, -B, W)\frac{(2+B)\sin(\pi(1-\frac{B}{2})W)}{(2-B)\sin(\pi(1+\frac{B}{2})W)}\exp(-\pi\sqrt{-1}BW)
\end{equation*}
at $A^{a}B^{b}W^{k}$ coincides with
\begin{equation*}
\sum_{0\le h \le l \le k-2}(-1)^{b-h}a_{h,l}Z_{a,b-h}(k-l).
\end{equation*}
Therefore, the coefficient of the right-hand side of \eqref{eq:SF3} at $A^{i-1}B^{r-i}W^k$ coincides with
\begin{equation*}
\sum_{0\le h \le l \le k-2}\bigl\{(-1)^{r-i-h}a_{h,l}Z_{i-1, r-i-h}(k-l)+(-1)^{k+i-1-h}a_{h,l}Z_{r-i, i-1-h}(k-l)\bigr\},
\end{equation*}
which completes the proof of Proposition \ref{prop:SFof1/2RS}.
\end{proof}

\begin{prop}\label{prop:SF_of_t-RS}
In $\cZ[2\pi \sqrt{-1}][t, A, B]\bbra{W}$, we have
\begin{align}\label{eq:gen_tRSMZV}
\begin{split}
&\sum_{\substack{\bk, \bl \in \cI \\ a\ge2}}\zeta^{t}_{\cRS}(\bk, a, \bl)A^{\dep(\bk)}B^{\dep(\bl)}W^{\wt(\bk)+a+\wt(\bl)}\\
&=F^{t}(A,-B,W)\frac{(1+tB)\sin(\pi(1-(1-t)B)W)}{(1-(1-t)B)\sin(\pi(1+tB)W)}\exp(-\pi\sqrt{-1}BW)\\
&\quad+F^{t}(B,-A,-W)\frac{(1+tA)\sin(\pi(1-(1-t)A)W)}{(1-(1-t)A)\sin(\pi(1+tA)W)}\exp(\pi\sqrt{-1}AW).
\end{split}
\end{align}
\end{prop}

\begin{proof}
This proposition is also proved in the same way as Lemma \ref{lem:decomp} by using Corollary \ref{cor:t-ham-zeta}.
\end{proof}

For non-negative integers $h \le l$, we define $a_{h,l}(t) \in \bQ[2\pi\sqrt{-1}][t]$ by
\begin{multline*}
a_{h,l}(t)
\coloneqq
\sum_{0 \le i \le j \le \lfloor l/2 \rfloor}\sum_{u\in \bZ}(-1)^{l-j+u}\frac{(2-2^{2i})B_{2i}}{(2i)!(2j-2i+1)!(l-2j)!}\\
\times\binom{2j-2i}{u}\binom{2i}{2j-(l-h)-u}t^{2j-(l-h)-u}(1-t)^{u}(2\pi\sqrt{-1})^{l}.
\end{multline*}
As mentioned in Remark \ref{rem:xi_RSMZV}, Bachmann--Takeyama--Tasaka's $\xi$-value coincides with the RSMZV. Thus, the following theorem can be regarded as an interpolation of their result \cite[Corollary 1.3]{BTT2}.
\begin{thm}\label{thm:t-BTT}
For positive integers $r<k$, we have
\begin{align*}
\sum_{\bk \in I_r(k,r)}\zeta^{t}_{\cRS}(\bk)
&=\sum_{j=0}^{r-1}\binom{k-1}{j}t^j(1-t)^{r-1-j}\zeta(k)\\
&\quad +\sum_{0\le h\le l \le k-2}\sum_{j=0}^{r-1-h}(-1)^{r-1-h+k}\binom{k-l-1}{j}(1-t)^{j}t^{r-1-h-j}a_{h,l}(t)\zeta(k-l).
\end{align*}
\end{thm}

\begin{proof}
By Proposition \ref{prop:SF_of_t-RS}, we have
\begin{align}\label{eq:gen_tRS_adm}
\begin{split}
&\sum_{\substack{\bk \in \cI \\ a\ge2}}\zeta^{t}_{\cRS}(\bk, a)A^{\dep(\bk)}W^{\wt(\bk)+a}\\
&=F^{t}(A, 0, W)+F^{t}(0, -A, -W)\frac{(1+tA)\sin(\pi(1-(1-t)A)W)}{(1-(1-t)A)\sin(\pi(1+tA)W)}\exp(\pi\sqrt{-1}AW).
\end{split}
\end{align}
We have
\begin{align}\label{eq:gen_Ft}
\begin{split}
F^{t}(A,0,W)
&=\sum_{n=1}^{\infty}\frac{W^2}{n(n-(1-(1-t)A)W))}\frac{(1-(1-(1-t)A)W)_n}{(1-(1+tA)W)_n}\\
&=\sum_{k=2}^{\infty}\sum_{r=0}^{k-2}\sum_{j=0}^{r}\binom{k-1}{j}t^{j}(1-t)^{r-j}\zeta(k)A^rW^k
\end{split}
\end{align}
and 
\begin{equation}\label{eq:gen_Ft_anti}
F^{t}(0, -A, -W)
=F^{1-t}(-A, 0, -W)
=\sum_{k=2}^{\infty}\sum_{r=0}^{k-2}(-1)^{r+k}\sum_{j=0}^{r}\binom{k-1}{j}(1-t)^jt^{r-j}\zeta(k)A^rW^k.
\end{equation}
By the Taylor expansion of $\sin X, \csc X=1/\sin X$ and $\exp X$, we have
\begin{equation}\label{eq:gen_sin}
\frac{(1+tA)\sin(\pi(1-(1-t)A)W)}{(1-(1-t)A)\sin(\pi(1+tA)W)}\exp(\pi\sqrt{-1}AW)
=\sum_{h=0}^{k}(-1)^ka_{h,k}(t)A^hW^k.
\end{equation}
Then, from \eqref{eq:gen_tRS_adm}, \eqref{eq:gen_Ft}, \eqref{eq:gen_Ft_anti} and \eqref{eq:gen_sin}, we see that the coefficient of the right-hand side of \eqref{eq:gen_tRSMZV} at $A^{r-1}B^0W^k$ coincides with the sum of the coefficient of 
\begin{equation*}
\sum_{i=0}^{k-2}\sum_{s=0}^{k-2-i}(-1)^{s+k-i}\sum_{j=0}^{s}\binom{k-i-1}{j}(1-t)^{j}t^{s-j}\zeta(k-i)A^s
\cdot \sum_{h=0}^{i}(-1)^{i}a_{h,i}A^h
\end{equation*}
at $A^{r-1}$ and 
\begin{equation*}
\sum_{j=0}^{r-1}\binom{k-1}{j}t^{j}(1-t)^{r-1-j}\zeta(k).
\end{equation*}
This completes the proof of Theorem \ref{thm:t-BTT}.
\end{proof}

Combining Theorems \ref{thm:wt SF3} and \ref{thm:t-BTT}, we obtain the weighted sum formula for 1/2-RSMZV of the same type as Theorem \ref{thm:wtd SF2}. Let $\delta_{a,b}$ denote Kronecker's delta.

\begin{cor}
For positive integers $r<k$, we have
\begin{align*}
&\sum_{(k_1, \ldots, k_r) \in I_r(k,r)}2^{k_r}\zeta^{\frac{1}{2}}_{\cRS}(k_1, \ldots, k_r)\\
&=\sum_{0 \le h \le l \le k-2}\Biggl\{(-1)^h\biggl(2^{k-l-1}-\sum_{j=0}^{r-1}\binom{k-l-1}{j}-\delta_{h,0}\biggr)\\
&\quad\qquad\qquad+(-1)^{k+r-1-h}\biggl(2^{k-l-1}+\sum_{j=0}^{r-1-h}\binom{k-l-1}{j}-1\biggr)
+2\sum_{j=0}^{r-1-h}\binom{k-1}{j}\Biggr\}a_{h,l}\frac{\zeta(k-l)}{2^{r-h}}.
\end{align*}
\end{cor}

\section*{Acknowledgments}
The authors would like to thank Professor Yasuo Ohno for his valuable comments on the paper \cite{OZ}. The authors also would like to thank Norihiko Namura and Taiki Watanabe for their helpful comments.


\begin{thebibliography}{99}
\bibitem{BKSYY}
H.~Bachmann, S.~Kadota, Y.~Suzuki, S. Yamasaki, S.~Yamamoto, 
\emph{Sum formulas for Schur multiple zeta values}, 
preprint (2023), arXiv:2302.03187.
\bibitem{BTT1}
H.~Bachmann, Y.~Takeyama, K.~Tasaka,
\emph{Cyclotomic analogues of finite multiple zeta values},
Compositio Math.\ \textbf{154} (2018), 2701--2721.
\bibitem{BTT2}
H.~Bachmann, Y.~Takeyama, K.~Tasaka,
\emph{Special values of finite multiple harmonic q-series at roots of unity}, 
IRMA Lectures in Mathematics and Theoretical Physics 31, ``Algebraic Combinatorics, Resurgence, Moulds and Applications (CARMA) Vol. 2" (2020), 1--18.
\bibitem{D}
J.~Dougall,
\emph{On Vandermonde's theorem and some more general expansions},
Proc.\ Edinburgh Math.\ Soc.\ \textbf{25}(1907), 114--132.
\bibitem{ELO}
M.~Eie, W-C.~Liaw, Y. L.~Ong, 
\emph{On generalizations of weighted sum formulas of multiple zeta values},
Int. J. Number Theory \textbf{9} (2013), 1185--1198.
\bibitem{FK}
K.~Fujita, Y.~Komori,
\emph{Weighted sum formulas for symmetric multiple zeta values},
The Ramanujan Journal \textbf{60} (2023), 141--155.
\bibitem{GX}
L.~Guo, B.~Xie,
\emph{Weighted sum formula for multiple zeta values}, 
J. Number Theory \textbf{129} (2009), 2747--2765.
\bibitem{Hi}
M.~Hirose,
\emph{Double shuffle relations for refined symmetric multiple zeta values},
Doc.\ Math.\ \textbf{25} (2020), 365--380.
\bibitem{HMS1}
M. Hirose, H. Murahara, S. Saito,
\emph{Generating functions for sums of polynomial multiple zeta values},
Tohoku Math. J. (2) \textbf{74} (3) (2022), 399--428. 
\bibitem{HMS2}
M. Hirose, H. Murahara, S. Saito,
\emph{Weighted sum formula for multiple harmonic sums modulo primes},
Proc.\ of Amer.\ Math.\ Soc.\ \textbf{147}(8) (2019), 3357--3366. 
\bibitem{Hof}
M. E. Hoffman,
\emph{The algebra of multiple harmonic series},
J. Algebra \textbf{194} (1997), 477--495.
\bibitem{IKOO}
K. Ihara, J. Kajikawa, Y. Ohno, J. Okuda,
\emph{Multiple zeta values vs. multiple zeta-star values},
J. Algebra \textbf{332} (2011), 187--208.
\bibitem{IKZ} K.~Ihara, M.~Kaneko, D.~Zagier, 
\emph{Derivation and double shuffle relations for multiple zeta values},
Compositio Math.\ \textbf{142} (2006), 307--338.
\bibitem{J} D.~Jarossay, 
Adjoint cyclotomic multiple zeta values and cyclotomic multiple harmonic values,
preprint (2019), arXiv:1412.5099v5.
\bibitem{Kad}
S.~Kadota, 
\emph{Certain weighted sum formulas for multiple zeta values with some parameters},
Comment. Univ.\ St.\ Pauli \textbf{66} (2017), 1--13.
\bibitem{Kam}
K.~Kamano,
\emph{Weighted sum formulas for finite multiple zeta values},
J. Number Theory \textbf{192} (2018), 168--180.
(2018)
\bibitem{L}
Z-h.\ Li,
\textit{Algebraic relations of interpolated multiple zeta values}, 
J. Number Theory \textbf{240} (2022), 439--470.
\bibitem{M1}
H.~Murahara,
\emph{A note on finite real multiple zeta values},
Kyushu J. Math.\ \textbf{70} (2016), 345--366.
\bibitem{M2}
H.~Murahara, 
\emph{A combinatorial proof of the weighted sum formula for finite and symmetric multiple zeta(-star) values},
Kobe journal of mathematics \textbf{38} (2021), 73--81.
\bibitem{MO}
H.~Murahara, M.~Ono, 
\emph{Yamamoto's interpolation of finite and symmetric multiple zeta values}, 
Tokyo Journal of Mathematics \textbf{44}(2) (2021), 285--312.
\bibitem{OZ} Y.~Ohno, W.~Zudilin, 
\emph{Zeta stars},
Commun.\ Number Theory Phys.\ \textbf{2} (2008), 325--347.
\bibitem{OSY} M.~Ono, S.~Seki, S.~Yamamoto,
\emph{Truncated $t$-adic symmetric multiple zeta values and double shuffle relations},
Res.\ number theory \textbf{7}, 15 (2021).
\bibitem{Re} C.~Reutenauer,
\emph{Free Lie Algebras}, Oxford Science Publications (1993).
\bibitem{SW} S. Saito, N. Wakabayashi,
\emph{Sum formula for finite multiple zeta values},
J. Math.\ Soc.\ Japan \textbf{67} (2015), 1069--1076.
\bibitem{S} S.~Seki,
\emph{Finite multiple polylogarithms}, Doctoral dissertation in Osaka University (2017).
\bibitem{Y} S.~Yamamoto, 
\emph{Interpolation of multiple zeta and zeta-star values},
J. Algebra \textbf{385} (2013), 102--114.
\bibitem{ZC} X.~Zhou, T.~Cai,
\emph{A generalization of a curious congruence on harmonic sums},
Proc.\ of Amer.\ Math.\ Soc.\ \textbf{135} (2007), no. 5, 1329--1333.

\end{thebibliography}
\end{document}